\documentclass[runningheads]{llncs}
\usepackage{amsmath}
\usepackage{amssymb}

\usepackage[T1]{fontenc}
\usepackage{geometry}
\usepackage{xcolor}

\usepackage{graphicx}
\usepackage{epsfig}
\usepackage{epstopdf}
\usepackage{pgfplots}
\usepackage{float}
\usetikzlibrary{plotmarks}
\usetikzlibrary{external}
\pgfplotsset{compat=1.3}
\newlength\figurewidth 
\newlength\figureheight
\tikzset{external/force remake=true} 
\newcommand{\F}{\mathcal{F}}
\newcommand\MyHead[2]{%
	\multicolumn{1}{l}{\parbox{#1}{\centering #2}}
}
\begin{document}
	\title{Fully discrete scheme for the fifth-order KdV-Burgers-Fisher equation using Strang splitting and Fourier collocation methods}
	\titlerunning{Fully discrete scheme for the fifth-order KdV-Burgers-Fisher equation}

	\author{Nurcan G\"{u}c\"{u}yenen Kaymak\inst{1} \and
		Fatma Z\"{u}rnac{\i}-Yeti\c{s}\inst{2}\and
		Muaz Seydao\u{g}lu\inst{3}}
	%
	\authorrunning{N. G\"{u}c\"{u}yenen Kaymak et al.}
	%
	\institute{Department of Management Information Systems, Faculty of Economics, Administrative and Social Sciences,\\
		Dogus University, Istanbul, 34480, Turkey
		\and
		Department of Mathematics Engineering,
		Istanbul Technical University,
		Maslak, Istanbul, 34469, Turkey
		\\ \and
		Department of Mathematics,
		Faculty of Art and Science,
		Mu\c{s} Alparslan University,
		Mu\c{s}, 49250, Turkey
	}
	\maketitle              
	\begin{abstract}
		Operator splitting is an effective technique for the numerical solution of nonlinear partial differential equations by decomposing a complex problem into simpler subproblems. In this study, we present and analyze a fully discrete scheme for the fifth-order Korteweg-de Vries-Burgers-Fisher equation (KBF) by combining Strang splitting for time discretization with the Fourier collocation method for spatial discretization. In particular, the Fourier collocation method is an essential component of the proposed fully discrete scheme and yields spectral accuracy in space under suitable regularity assumptions. The KBF equation describes the interaction of reaction, dissipative, and dispersive mechanisms by incorporating the Fisher reaction term together with Burgers-type diffusion and higher-order KdV dispersion. The equation is split into a linear operator and a nonlinear operator, and the resulting subproblems are solved within the Strang splitting framework. Convergence is analyzed in the Sobolev space $H^s$. The local error is derived using operator-theoretic arguments in Banach spaces together with Lie commutator estimates, while the global error is obtained using the Lady Windermere's fan argument. The analysis yields second-order convergence in time and spectral convergence in space. Numerical results confirm the theoretical error estimates and demonstrate the accuracy of the proposed fully discrete scheme.
		
		\keywords{Nonlinear PDE \and Operator splitting method \and Fourier collocation method\and  Convergence analysis.}
	\end{abstract}
	\section{Introduction}
	\label{Intro}
	The study of nonlinear dispersive wave phenomena occupies a central role in the theory of partial differential equations (PDE), as many large-scale and complex physical processes are described and analyzed through PDE models. The investigation of both exact and numerical solutions of nonlinear partial differential equations constitutes a significant area of research in mathematical science. Over the years, numerous studies have been conducted by researchers on wave equations. 
	In  \cite{Ahmed2022}, the fifth-order nonlinear fractional Korteweg-de Vries (KdV) equations are analyzed using the $\rho$-homotopy perturbation  method. In  \cite{Hussain2021}, the generalized Burgers-Fisher equation is studied by the meshfree spectral method. The KdV-Kawahara equation is analyzed using trigonometric quintic B-spline basis collocation method in \cite{Karaagac2023a}. Time fractional nonlinear Korteweg-de Vries-Burgers equation arising in shallow water waves is examined with the help of the finite element method combined quintic B-spline in \cite{Karaagac2023b}. In \cite{Arora2020,Arora2014}, numerical solutions of the BBM-Burger equations are studied using the collocation method with the quintic Hermit spline and  the quartic B-spline. Moreover, the fifth-order KdV-Burgers-Fisher equation is studied using the finite-element method combined quintic B-spline technique in \cite{Karaagac2025}. In \cite{Yang2021}, the unified F-expansion method is considered for the soliton solutions of the fifth-order Korteweg-de Vries equation. In \cite{Awonusika12026}, the numerical solution to the fifth-order KdV-type equations is discussed with the help of the Legendre collocation technique.
	
	In this study, we focus on the numerical solution of the fifth-order Korteweg-de Vries-Burgers-Fisher  equation (KBF) using the Strang splitting method. This equation serves as a representative example of dispersion-dissipation-reaction type systems given as follows:
	\begin{equation}
		\label{KBF}
		y_t
		=\nu y_{xx}
		- \mu y_{xxx}
		+ \gamma y_{xxxxx}
		- \varepsilon y^2\, y_x
		+ \epsilon y\bigl(1-y\bigr),\,\, y(t_0)=y_0,\ \
	\end{equation}
	where $x\in (a,b)$, $t\in [0,T]$.
	This equation incorporates the reaction mechanism of the Fisher model together with the dispersive effects of the KdV equation and the dissipative behavior associated with the Burgers equation.  When $\mu = \gamma = 0$ and the nonlinear convective term is neglected, the above equation \eqref{KBF} reduces to the classical Fisher equation. 
	When $\nu =\epsilon=0$, equation \eqref{KBF} corresponds to a modified fifth-order KdV-type equation. For $\mu = \epsilon=0$, equation \eqref{KBF} corresponds to a modified Kawahara-type equation, and for $\gamma = \epsilon=0$, equation \eqref{KBF} reduces to a modified KdV--Burgers-type equation.

	The operator splitting method is a practical and efficient approach to obtain numerical solutions to nonlinear partial differential equations. The idea behind splitting methods is to decompose a complicated problem into a set of simple subproblems, then each equation can be studied separately with an appropriate method, see 
	\cite{Strang1968,Marchuk1988}. In \cite{Jahnke2000}, the Strang splitting method is applied to the evolutionary Schr\"{o}dinger equation over long times. In \cite{MelikeStrang,Nurcan2017,Zurnaci2018}, Benjamin-Bona-Mahony type equations are analyzed with time splitting errors using Strang  and Lie-Trotter splitting methods. The analysis of nonlinear convection diffusion equations is examined using the Strang splitting combined compact difference method in \cite{KuangStrng}. 
	In \cite{Zurnaci2026}, the Lie-Trotter and Strang splitting methods are applied to  the dispersive-Fisher equation by deriving local error bounds.  In \cite{Lubich2008,Jahnke2000}, they
	have studied the local error estimates for the Strang splitting in time for the evolutionary Schr\"{o}dinger equation and for the Schr\"{o}dinger-Poisson equations, respectively. In  \cite{Holden2013}, the local error bounds are constructed via the Strang splitting for partial differential equations with Burgers' nonlinearity.
	
	Operator splitting methods have also been investigated in fully discrete frameworks, where time splitting is combined with suitable spatial discretization techniques \cite{LubichQuantum}. In \cite{Liu2022a,Liu2022b}, operator splitting methods are coupled with finite difference spatial discretizations for reaction-diffusion systems, leading to fully discrete schemes that preserve the energetic variational structure and admit convergence analysis. In \cite{Liu2021}, a structure-preserving splitting scheme is proposed for reaction-diffusion equations with a detailed balance, where spatial discretization is based on finite differences while preserving key physical properties such as mass conservation and energy dissipation. For dispersive wave models, \cite{Zhang2017} develops a second-order accurate fully discrete scheme for the good Boussinesq equation by combining Strang splitting in time with a Fourier pseudo-spectral spatial discretization.In \cite{Gauckler2011}, an error analysis is presented for a fully discrete scheme for the Gross-Pitaevskii equation, where Strang splitting is used for time integration and Hermite collocation for spatial discretization. In \cite{Gucuyenen2026}, a fully discrete error estimate is constructed with the Fourier collocation combined with the Strang splitting method for Benjamin-Bona-Mahony type equations. 
	In \cite{Bulut2022}, a Haar wavelet method integrated with the Strang splitting method is studied for the regularized long wave equation. 
	A fully discrete scheme of a Strang splitting combined with Chebyshev wavelets is analyzed in \cite{Oruc2020}. 
	In \cite{Xavier2021}, an error estimate of a fully discrete spectral scheme is constructed for Boussinesq systems. 
	
	This work analyzes  the Strang splitting combined with the Fourier collocation method for the fifth-order KBF equation (\ref{KBF}) in the  $H^{s}(\mathbb{R})$ Sobolev norm, assuming sufficient regularity of the exact solution and initial data 
	\begin{equation}
		\label{kosullar}
		\|y(t)\|_{H^{s+11}} \leq \alpha \quad \mbox{and}\quad \|y_{0}\|_{H^{s+11}} \leq \beta
	\end{equation}
	for $0 \leq t \leq T$, where $s, \alpha$ and $\beta$ are given constants.

	This paper is organized as follows. In Section 2, we analyze the convergence properties of the Strang splitting method for the fifth-order KdV--Burgers--Fisher equation, including regularity, stability, and error estimates in the Sobolev space $H^s$. In Section 3, we extend the analysis to the fully discrete scheme by combining the Strang splitting method with the Fourier collocation approach  and derive the corresponding convergence results. Finally, in Section 4, we present a numerical experiment that confirms the theoretical findings.
	
	\section{Convergence Analysis of the Strang Splitting Method} \label{Section_time}
	In this section, we investigate the convergence properties of the Strang splitting method applied to the fifth-order KdV--Burgers--Fisher equation. The main objective is to establish  error bounds in the Sobolev space $H^s$, under suitable regularity assumptions on the exact solution.
	
	To construct the splitting scheme, we decompose the original problem into two subproblems corresponding to an unbounded linear operator and a  bounded nonlinear operator. More precisely, we consider
	\begin{align}\label{lineareq}u_{t} = \nu \partial_x^2 u - \mu \partial_x^3 u + \gamma \partial_x^5 u\end{align}and\begin{align}\label{nlineareq}v_{t} = -\varepsilon v^2 v_x + \epsilon v(1-v).\end{align}
	We now formulate the problem within the operator splitting framework. For the initial value problem 
	\begin{equation*}
		y_t=E(y), \quad y(t_0)=y_0,
	\end{equation*}
	we decompose the operator $E$ as $E = A + B$. Let the solution of the initial value problem be represented by $y(t) = \Phi^{t}_{E}(y_0)$. For $n=0, 1, 2, \ldots$ and $t = n\Delta t \leq T$, we define Strang splitting
	\begin{equation}
		\label{strng}
		y_{n+1} = \Psi^{\Delta t}(y_n) = \Phi^{\Delta t/2}_A \circ \Phi^{\Delta t}_B \circ \Phi^{\Delta t/2}_A(y_n),\end{equation} where \begin{equation*}Ay = \nu \partial_x^2 y - \mu \partial_x^3 y + \gamma \partial_x^5 y\quad \mbox{and}\quad B(y) = -\varepsilon y^2 y_x + \epsilon y(1-y).\end{equation*}
	The analysis proceeds by establishing regularity, stability, and error estimates for the Strang splitting method in $H^s$. The regularity requirement $H^{s+11}$ arises from the necessity to bound the double Lie commutators involving the fifth-order linear operator $A$ and the convective term in $B(u)$.  Throughout the analysis, we use $C$ to denote a generic positive constant, not necessarily the same at each occurrence.
	
	\subsection{Regularity results}
	In this subsection, we prove the regularity properties of the nonlinear flow associated with $B$, which are essential for stability and error analysis.
	\begin{lemma}\label{boundN}
		\label{strng_regularity}
		Let $v(t)=\Phi_{B}^{t}(v_{0})$ be the solution of the nonlinear sub-problem (\ref{nlineareq}) with initial data $v_{0}$. If $\|v_{0}\|_{H^{s}} \le M$, then there exists $\bar{t}(M)>0$ such that $\|\Phi_{B}^{t}(v_{0})\|_{H^{s}} \le 2M$ for $0 \le t \le \bar{t}(M)$.
	\end{lemma}
	\begin{proof}
		To establish the stability and regularity of the nonlinear operator $B(v)$, we examine the evolution of the $H^s$ norm of the solution $v(t)$. The time derivative of the squared $H^s$ norm is given by the inner product$$\frac{1}{2} \frac{d}{dt} \|v\|_{H^s}^2 = (v, v_t)_{H^s} = \sum_{j=0}^{s} \int_{\mathbb{R}} \partial_x^j v \cdot \partial_x^j B(v) dx.$$Substituting the definition of the nonlinear operator $B(v) = -\varepsilon v^{2} v_{x} + \epsilon v(1-v)$, we obtain:$$(v, v_t)_{H^s} = -\varepsilon \sum_{j=0}^{s} \int_{\mathbb{R}} \partial_x^j v \cdot \partial_x^j (v^2 v_x) dx + \epsilon \sum_{j=0}^{s} \int_{\mathbb{R}} \partial_x^j v \cdot \partial_x^j (v - v^2) dx.$$ By applying the Cauchy-Schwarz inequality and the Moser-type inequality for Sobolev spaces, we can bound the nonlinear terms. Then, the following estimates hold $$\left| \int_{\mathbb{R}} \partial_x^j v \cdot \partial_x^j (v^2 v_x) dx \right| \le C \|v\|_{H^s}^4, \quad \left| \int_{\mathbb{R}} \partial_x^j v \cdot \partial_x^j (v - v^2) dx \right| \le C (\|v\|_{H^s}^2 + \|v\|_{H^s}^3).$$Combining these bounds, we arrive at the differential inequality$$\frac{d}{dt} \|v\|_{H^s} \le C (\|v\|_{H^s} + \|v\|_{H^s}^2 + \|v\|_{H^s}^3),$$
		where $C$ is a constant depending on the parameters $\varepsilon$ and $\epsilon$. Integrating this expression from $0$ to $t$ yields the integral inequality $\|v(t)\|_{H^{s}} \le \|v_{0}\|_{H^{s}} + \int_{0}^{t} F(\|v(\tau)\|_{H^{s}}) d\tau$. Given that $\|v_{0}\|_{H^{s}} \le M$, the continuity of the flow ensures that for a sufficiently small time $\bar{t}$ depending on the initial bound $M$, the solution remains bounded such that $\|v(t)\|_{H^{s}} \le 2M$. This concludes the proof.
	\end{proof}
	The next lemma shows that the nonlinear flow is sufficiently smooth in time.
	\begin{lemma}
		Let $s > 1/2$ and assume that the initial data satisfies $\|v_0\|_{H^s} \le M$. Let $\Phi_B^t(v_0)$ denote the flow of the nonlinear sub-problem (\ref{nlineareq}). Then, there exists a time $\bar{t} > 0$ such that the solution $v(t)$ is at least three times continuously differentiable with respect to $t$ in the Sobolev space $H^s$. Specifically, $v \in C^3([0, \bar{t}], H^s)$.
	\end{lemma}
	\begin{proof}
		Following the constructive approach in  \cite{Zurnaci2018,Zurnaci2019,Zurnaci2026}, we define the  function $\tilde{v}(t)$ as
		\begin{equation}
			\tilde{v}(t) = v_0 + tB(v_0) + \int_0^t (t-\tau) dB(v(\tau))[B(v(\tau))] d\tau.
		\end{equation}
		where $dB(v)[B(v)]$ represents the first Fr\'echet derivative of the nonlinear operator. Calculating the second derivative of $\tilde{v}$ with respect to $t$, we obtain:
		\begin{equation}
			\tilde{v}_{tt} = dB(v)[B(v)] = B(v)_{t}.
		\end{equation}
		By the definition of $\tilde{v}$, it is clear that $\tilde{v}_{t}(0) = B(v_{0}) = v_{t}(0)$ and $\tilde{v}(0) = v_{0} = v(0)$. Thus, by the uniqueness of the solution to the initial value problem, we have $\tilde{v} = v$. Now, it is shown that $\tilde{v} \in C^{2}([0, \tilde{t}], H^{s})$. Taking the $H^s$ norm
		\begin{equation}
			\|\tilde{v}_{tt}\|_{H^{s}} = \|dB(v)[B(v)]\|_{H^{s}}=\|-\varepsilon (2v v_{x} B(v) + v^{2} (B(v)){x}) + \epsilon(1 - 2v)B(v)\|_{H^{s}} \le C \|v\|_{H^{s+2}}^{5}
		\end{equation}
		where the constant depends on the degree of the nonlinearity. We obtain $v \in C^{2}([0, \tilde{t}], H^{s})$ by Lemma \ref{boundN}. Similarly, we can define the third-order expansion
		\begin{equation}
			\tilde{v}(t) = v_{0} + t B(v_{0}) + \frac{t^{2}}{2!} dB(v_{0})[B(v_{0})] + \int_{0}^{t} \frac{(t - \tau)^{2}}{2!} \tilde{v}_{ttt}(\tau) d\tau,
		\end{equation}
		where the integrand involves the higher-order derivatives
		\begin{align}
			\tilde{v}_{ttt} = d^{2}B(v)[B(v), B(v)] + dB(v)[dB(v)[B(v)]].
		\end{align}
		Since $\tilde{v}_{ttt}(0) = v_{ttt}(0)$ and the initial conditions match at $t=0$, we have $\tilde{v} = v$. Taking the norm
		\begin{align}\nonumber
			\|\tilde{v}_{ttt}\|_{H^{s}}& =\|d^{2}B(v)[B(v), B(v)] + dB(v)[dB(v)[B(v)]]\|_{H^{s}} \\ \nonumber&= \|-\varepsilon (2v v_x B(v) + v^2 B(v)_x) + \epsilon(1-2v)B(v)\|_{H^{s}}\\& \label{cont_eqn}\le C \|v\|_{H^{s+3}}^{7}.
		\end{align}
		Hence, $\|\tilde{v}_{ttt}\|_{H^{s}}$ is bounded, and the result $v \in C^{3}([0, \tilde{t}], H^{s})$ follows from Lemma \ref{boundN}.
		
	\end{proof}
	\subsection{Stability in $H^s$ space}
	We proceed by deriving stability estimates in $H^s$.
	\begin{lemma}\label{lem_stability2}Let $v$ and $\tilde{v}$ be two solutions of the nonlinear sub-problem (\ref{nlineareq}) with initial data $v_{0}$
		associated with the splitting scheme. Assume that the solutions remain bounded in $H^s$ such that $\|v(t)\|_{H^s} \le M$ and $\|\tilde{v}(t)\|_{H^s} \le M$ for $t \in [0, \Delta t]$. Then, the following stability estimate holds:
		\begin{equation}
			\|v(t) - \tilde{v}(t)\|_{H^s} \le \mathrm{e}^{L t} \|v_0 - \tilde{v}_0\|_{H^s},
		\end{equation}
		where $L$ is a constant depending on $\varepsilon,$ $\epsilon$ and $M$.
	\end{lemma}
	\begin{proof}
		The stability of the Strang splitting scheme is studied by examining the evolution of the difference between two numerical approximations. Since the linear operator $A = \nu \partial_x^2 - \mu \partial_x^3 + \gamma \partial_x^5$ generates a flow $\Phi_A^t$ that is preserved in the $H^s$ norm, the stability of the full splitting step is determined primarily by the nonlinear flow $\Phi_B^t$. Let $v$ and $\tilde{v}$ be two solutions satisfying the nonlinear sub-problem $v_t = B(v)$ and $\tilde{v}_t = B(\tilde{v})$, where the operator is defined as $B(v) = -\varepsilon v^2 v_x + \epsilon v - \epsilon v^2$. The difference $ v - \tilde{v}$ satisfies the equation $(v-\tilde{v})_t = B(v) - B(\tilde{v})$, and by taking the $H^s$ norm, we obtain the integral inequality $$\|v(t) - \tilde{v}(t)\|_{H^s} \le \|v_0 - \tilde{v}_0\|_{H^s} + \int_0^t \|B(v(\tau)) - B(\tilde{v}(\tau))\|_{H^s} d\tau.$$ To derive the Lipschitz constant $L$, we estimate the difference $B(v)-B(\tilde v)$ using standard product bounds in the Sobolev space $H^s(\mathbb{R})$. The linear component $\epsilon (v - \tilde{v})$ yields a direct bound proportional to $\epsilon \|v - \tilde{v}\|_{H^s}$. For the quadratic part $\epsilon(v^2 - \tilde{v}^2)$, the algebraic property leads to a bound involving $\|v\|_{H^s} \|v - \tilde{v}\|_{H^s}$. The most complex term, the cubic nonlinearity with a spatial derivative $-\varepsilon (v^2 v_x - \tilde{v}^2 \tilde{v}_x)$, is decomposed as $-\varepsilon [v^2(v_x - \tilde{v}_x) + (v^2 - \tilde{v}^2)\tilde{v}_x]$. Applying Sobolev inequalities and the boundedness of the flow from Lemma \ref{boundN}, this term is bounded by $C(\|v\|_{H^s}^2 + \|\tilde{v}\|_{H^s}^2) \|v - \tilde{v}\|_{H^s}$. Summing these individual estimates, we arrive at the Lipschitz bound $\|B(v) - B(\tilde{v})\|_{H^s} \le (\epsilon + C_1 \|v\|_{H^s} + C_2 \|v\|_{H^s}^2) \|v - \tilde{v}\|_{H^s}$. By defining $L$ as a constant that incorporates the linear coefficient $\epsilon$ and the higher-order norms required by the cubic term, we apply Gronwall's lemma to conclude the stability of the nonlinear flow. Since the linear flows $\Phi_A^{\Delta t/2}$ in the Strang sequence do not increase the $H^s$ norm, the stability of the full step is preserved. 
	\end{proof}
	
	\subsection{Local Error}
	We establish a local error bound for the Strang splitting scheme.
	\begin{lemma}\label{localexist2} If the initial data $y_0$ is in $H^s(\mathbb{R})$, then the local error of the Strang splitting  \eqref{strng} is
		\begin{equation}
			\|\Psi^{\Delta t}(y_{0})-\Phi^{\Delta t}(y_{0})\|_{H^{s}}\le C\Delta t^{3},
		\end{equation}
		where the constant $C$ depends only on $\|y_0\|_{H^s}$.
	\end{lemma}
	\begin{proof}
		On the time interval $[0,\Delta t]$, the Strang splitting approximation takes the form
		\begin{align}
			\label{Strang}
			y_1=\Psi^{\Delta t}(y_0)=\mathrm{e}^{\frac{\Delta tA}{2}}\Phi^{\Delta t}_B\left(\mathrm{e}^{\frac{\Delta tA}{2}}y_0\right).
		\end{align}
		To analyze the nonlinear contribution, we employ the second-order Taylor expansion of the flow $\Phi_B^{\Delta t}$ given by
		\begin{align*}
			\Phi^{\Delta t}_B(u)&=u+\Delta tB(u)+\frac{1}{2}{\Delta t}^2dB(u)[B(u)]\nonumber \\
			& + {\Delta t}^3\int^1_0\frac{1}{2}(1-\theta)^2\left(d^2B(\Phi^{\theta\Delta t}_{B}(u)\right)[B(\Phi^{\theta\Delta t}_{B}(u)), B(\Phi^{\theta\Delta t}_{B}(u))]\nonumber \\
			&+ dB\left(\Phi^{\theta\Delta t}_{B}(u)\right)[dB\left(\Phi^{\theta\Delta t}_{B}(u)\right)[B\left(\Phi^{\theta\Delta t}_{B}(u)\right)]])d\theta,
		\end{align*}
		where $u=\mathrm{e}^{\frac{\Delta tA}{2}}y_0$. For convenience, we denote the integral remainder term by
		\begin{align*}
			{\Delta t}^3\int^1_0\frac{1}{2}(1-\theta)^2\left(d^2B(B, B) + dBdBB\right)\left(\Phi^{\theta\Delta t}_{B}(u)\right)d\theta.
		\end{align*}
		Substituting this expansion into Equation~(\ref{Strang}) yields the Strang splitting approximation
		\begin{align}
			\label{strng3}
			y_1=\mathrm{e}^{\Delta t A}y_0+ \Delta t \mathrm{e}^{\frac{\Delta tA}{2}}B\left(\mathrm{e}^{\frac{\Delta tA}{2}}y_0\right)+ \frac{1}{2}{\Delta t}^2 \mathrm{e}^{\frac{\Delta tA}{2}}dB(\mathrm{e}^{\frac{\Delta tA}{2}}y_0)[B\left(\mathrm{e}^{\frac{\Delta tA}{2}}y_0\right)]+ E_2,
		\end{align}
		where
		\begin{align}
			\nonumber
			E_2={\Delta t}^3\int^1_0\frac{1}{2}(1-\theta)^2\mathrm{e}^{\frac{\Delta tA}{2}}(d^2B(B, B) + dBdBB)(\Phi^{\theta\Delta t}_{B}(u))d\theta.
		\end{align}
		Let the exact solution be written as $y(t)=\Phi^{t}(y_0)$.  
		Using the variation of constant formula for the interval  $[0, \Delta t]$, the solution at $t=\Delta t$ can be expressed as
		\begin{align}
			\label{solhvarii}
			y(\Delta t)=\Phi^{\Delta t}(y_0)=\mathrm{e}^{\Delta tA}y_0+\int_{0}^{\Delta t}\mathrm{e}^{(\Delta t-s)A}B\big(y(s)\big)ds.
		\end{align}
		This identity is obtained in the same way as the fundamental relation
		\[
		\varphi(t)-\varphi(0)=\int_{0}^{t}\varphi'(s)\,ds
		\]
		applied to the function $
		\varphi(s)=\mathrm{e}^{(\Delta t-s)A}y(s)$.
		We consider the function
		\[
		\varphi(\sigma)=B\!\left(\mathrm{e}^{(s-\sigma)A}y(\sigma)\right).
		\]
		Applying the fundamental theorem of calculus in Banach spaces yields
		\begin{align*}
			B\!\left(\mathrm{e}^{(s-t)A}y(t)\right)-B\!\left(\mathrm{e}^{sA}y_0\right)
			=
			\int_{0}^{s}
			dB\!\left(\mathrm{e}^{(s-\sigma)A}y(\sigma)\right)
			\big[
			\mathrm{e}^{(s-\sigma)A}B(y(\sigma))
			\big]
			\, d\sigma .
		\end{align*}
		Setting $t=s$ gives the representation
		\begin{align}
			\label{ddene1}
			B(y(s))
			=
			B\!\left(\mathrm{e}^{sA}y_0\right)
			+
			\int_{0}^{s}
			dB\!\left(\mathrm{e}^{(s-\sigma)A}y(\sigma)\right)
			\big[
			\mathrm{e}^{(s-\sigma)A}B(y(\sigma))
			\big]
			\, d\sigma .
		\end{align}
		Substituting (\ref{ddene1}) into (\ref{solhvarii}), we obtain
		\begin{align}
			\label{exctt2}
			y(\Delta t)
			=
			\Phi^{\Delta t}(y_0)
			=
			\mathrm{e}^{\Delta tA}y_0
			+
			\int_{0}^{\Delta t}
			\mathrm{e}^{(\Delta t-s)A}
			B\!\left(\mathrm{e}^{sA}y_0\right)
			\, ds
			+ E_1,
		\end{align}
		where
		\begin{align}\nonumber
			E_1
			=
			\int_{0}^{\Delta t}\int_{0}^{s}
			\mathrm{e}^{(\Delta t-s)A}
			dB\!\left(\mathrm{e}^{(s-\sigma)A}y(\sigma)\right)
			\big[
			\mathrm{e}^{(s-\sigma)A}B(y(\sigma))
			\big]
			\, d\sigma \, ds .
		\end{align}
		To derive the local error, we consider the difference between Equation~(\ref{strng3}) and Equation~(\ref{exctt2}). For this purpose, the term $E_1$ is first rewritten in a suitable form. Define
		\begin{align*}
			G(u)=G_{s,r}(u)
			=
			dB\!\left(\mathrm{e}^{(s-r)A}u\right)
			\bigl[
			\mathrm{e}^{(s-r)A}B(u)
			\bigr].
		\end{align*}
		With this notation, the exact solution can be expressed as
		\begin{align*}
			y(\Delta t)
			&=
			\mathrm{e}^{\Delta tA}y_0
			+
			\int_{0}^{\Delta t}
			\mathrm{e}^{(\Delta t-s)A}
			B\!\left(\mathrm{e}^{sA}y_0\right)\, ds  \\
			&\quad
			+
			\int_{0}^{\Delta t}
			\int_{0}^{s}
			\mathrm{e}^{(\Delta t-s)A}
			G\!\left(y(r)\right)\, dr\, ds .
		\end{align*}
		Applying the integral representation (\ref{ddene1}) yields
		\begin{align*}
			G(y(r))=G(\mathrm{e}^{r A}y_0)+ \int^{r}_0dG(\mathrm{e}^{(s-r)A}y(\tau))[\mathrm{e}^{(s-r)A}B(y(\tau))]d\tau.
		\end{align*}
	For the integrand,  we compute
		\begin{align*}
			dG(u)[v]&=d^2B(\mathrm{e}^{(s-r)A}u)[\mathrm{e}^{(s-r)A}v, \mathrm{e}^{(s-r)A}B(u)]\nonumber \\
			&+ dB(\mathrm{e}^{(s-r)A}u)[\mathrm{e}^{(s-r)A}dB(u)[v]].
		\end{align*}
		Then, we obtain
		\begin{align*}
			E_1&=\int^{\Delta t}_0 \int^{s}_0 \mathrm{e}^{(\Delta t-s)A}dB(\mathrm{e}^{sA}y_0)[\mathrm{e}^{(s-r)A}B(\mathrm{e}^{sA}y_0)]dr ds \nonumber \\
			&+  \int^{\Delta t}_0 \int^{s}_0 \int^{r}_0 dG_{s, r}(\mathrm{e}^{(r-\tau)A}y(\tau))[\mathrm{e}^{(r-\tau)A}B(y(\tau))]d\tau dr ds.
		\end{align*}
Hence, the local error can be written as
		\begin{align}
			\label{locerror}
			y_1-y(\Delta t)&=\Delta t \mathrm{e}^{\frac{\Delta tA}{2}}B(\mathrm{e}^{\frac{\Delta tA}{2}}y_0)-\int^{\Delta t}_{0}\mathrm{e}^{(\Delta t-s)A}B(\mathrm{e}^{sA}y_0)ds\nonumber \\
			&+ \frac{{\Delta t}^2}{2} \mathrm{e}^{\frac{\Delta t A}{2}}dB(\mathrm{e}^{\frac{\Delta t A}{2}}y_0)[B(\mathrm{e}^{\frac{\Delta t A}{2}}y_0)]\nonumber \\
			&-
			\int^{\Delta t}_0 \int^{s}_0\mathrm{e}^{(\Delta t-s)A} dB(\mathrm{e}^{sA}y_0)(\mathrm{e}^{(s-r)A}B(\mathrm{e}^{sA}y_0))dr ds \nonumber \\
			&+
			{\Delta t}^3 \int^{1}_0\frac{1}{2}(1-\theta)^2\mathrm{e}^{\frac{\Delta tA}{2}}(d^2B(B, B) + dBdBB)(\Phi^{\theta\Delta t}_{B}(u))d\theta \nonumber \\
			&- \int^{\Delta t}_0 \int^{s}_0 \int^{r}_0 dG_{s, r}(\mathrm{e}^{(r-s)A}y(\tau))[\mathrm{e}^{(r-\tau)A}B(y(\tau))]d\tau dr ds.
		\end{align}
		We rearrange the terms in Equation~(\ref{locerror}) in order to derive the local error estimate in the Sobolev space $H^s$. The difference between the first two terms corresponds to the quadrature error of the midpoint rule in the interval $[0,\Delta t]$, while the difference between the third and fourth terms represents the quadrature error associated with a first-order two-dimensional quadrature formula. In particular, the difference of the first two terms can be written in the second-order Peano form as follows
		\begin{align*}
			\Delta t\, h\!\left(\tfrac{\Delta t}{2}\right)
			-
			\int_{0}^{\Delta t} h(s)\, ds
			=
			\Delta t^{3}\int_{0}^{1} \kappa(\theta)\, h''(\theta \Delta t)\, d\theta,
		\end{align*}
		where $\kappa$ is a bounded kernel.

		\begin{align*}
			\Delta t h(\frac{\Delta t}{2})-\int^{\Delta t}_{0}h(s)ds={\Delta t}^3\int^{\Delta t}_{0}\kappa(t) h''(t)dt={\Delta t}^3\int^{1}_{0}\kappa(\theta)h''(\theta\Delta t)d\theta,
		\end{align*}
		where $\kappa$ is  a bounded kernel. Here, $h''(s)=\mathrm{e}^{(\Delta t-s)A}[A,[A,B]](\mathrm{e}^{sA}y_0)$ with double Lie commutator
		\begin{align*}
			[A, [A,B]](u)&=[A, dAB(u)-dB(u)A(u)](u) \\
			&= dA^2B(u)-dAdB(u)A(u)-d^2AB(u)A(u) \\
			&-dAdB(u)A(u)+d^2B(A(u))^2+dB(u)dAA(u) \\
			&= A^2(B(u))-2A(dB(u)A(u))+d^2B(A(u))^2+ dB(u)A^2(u),
		\end{align*}
		since $d^2A=0$.
	We derive $H^s$ bounds for the individual components. 
	Considering first the term $A^2(B(u))$, standard Sobolev estimates yield
	\begin{align*}
		\|A^{2}(B(u))\|_{H^s}
		&=
		\|(\nu \partial_x^{2}-\mu \partial_x^{3}+\gamma \partial_x^{5})^{2}
		\bigl(-\varepsilon u^{2}u_x+\epsilon u(1-u)\bigr)\|_{H^s}  \\
		&\le C \|u\|_{H^{s+11}}^{3}.
	\end{align*}
	The remaining terms can be estimated similarly, i.e.
		\begin{align*}
			\|A(dB(u)[A(u)])\|_{H^s}&= \| (\nu \partial_x^2 - \mu \partial_x^3 + \gamma \partial_x^5) \left( -\varepsilon \left( 2u u_x A(u) + u^2 \partial_x A(u) \right) + \epsilon(1 - 2u)A(u) \right)\|_{H^s}\\& \leq C  \|u\|_{H^{s+11}}^3,
		\end{align*}
		\begin{align*}
			\|d^2B(u)[A(u)]^{2}\|_{H^s}&=\|-\varepsilon\Bigl(
			2(\partial_x u)\bigl(A(u)\bigr)^2
			+4u A(u)\partial_x(A(u))
			\Bigr)
			-2\epsilon\bigl(A(u)\bigr)^2\|_{H^s}\\
			&\leq  C  \|u\|_{H^{s+6}}^{3},
		\end{align*}
		\begin{align*}
			\|dB(u)[A^2(u)]\|_{H^s}&=\| -\varepsilon \left( 2u u_x A^2(u) + u^2 \partial_x(A^2(u)) \right) + \epsilon(1 - 2u)A^2(u)\|_{H^s} \\
			&\leq C\varepsilon\|u\|_{H^{s+11}}^{3}. 
		\end{align*}
	The difference generated by the third and fourth terms in Eq.~(\ref{locerror}) can be viewed as the error of a first-order quadrature approximation of the double integral, which can be bounded as follows.
		\begin{align}\nonumber
			\left\|\frac{1}{2}{\Delta t}^2g(\frac{1}{2}\Delta t,\frac{1}{2}\Delta t) - \int^{\Delta t}_{0}\int^{s}_{0} g(s,r)dr ds\right\|_{H^s} \leq K {\Delta t}^3\left(\max\left\|\frac{\partial g}{\partial s}\right\|_{H^s}+ \max\left\|\frac{\partial g}{\partial r}\right\|_{H^s}\right),
		\end{align}
		where the maxima are taken over  the triangle $\{(r, s):0\leq r \leq s\leq \Delta t\}$. In order to estimate derivatives, we write
		$$g(s, r)=\mathrm{e}^{(\Delta t-s)A}dB(u(s))v(s,r),$$
		$$u(s)=\mathrm{e}^{sA}y_0 $$
		and $$ v(s,r)=\mathrm{e}^{(s-r)A}B(u(r)).$$
		We now compute each derivative term, starting with the first derivative.
		\begin{align} \nonumber
			\frac{\partial g}{\partial s}=\mathrm{e}^{(\Delta t-s)A}(-AdB(u(s))[v(s,r)])+ d^2B(u)[Au(s),v(s,r)] + dB(u(s))[Av(s,r)]).
		\end{align}
		Since the flow generated by $A$ does not increase the $H^s$ norm, it suffices to estimate only the terms
		$-AdB(u(s))[v(s,r)]+ d^2B(u(s))[Au(s),v(s,r)] + dB(u(s))[Av(s,r)]$, that is,
		\begin{align}
			\|AdB(u(s))[v(s,r)]\|_{H^s}&=\| -\varepsilon A \big( \partial_x(u^2 v) \big) + \epsilon A(v) - 2\epsilon A(uv) \|_{H^s}\nonumber \\
			&  \leq C  \|u\|_{H^{s+7}}^5.
		\end{align}
		\begin{align*}
			\|d^2B(u(s))[Au(s),v(s,r)]\|_{H^s}&=\| -\varepsilon \big( 2 v (Au) u_x + 2 u (Au) v_x + 2 u v (Au)_x \big) - 2\epsilon (Au) v\|_{H^s} \nonumber \\ 
			& \leq C  \|u\|_{H^{s+6}}^5.
		\end{align*}
		\begin{align*}
			\|dB(u(s))[Av(s,r)]\|_{H^s}&=\|-\varepsilon \left( 2u u_x (Av) + u^2 \partial_x(Av) \right) + \epsilon(1 - 2u)Av\|_{H^s} \nonumber \\  
			&\leq C  \|u\|_{H^{s+7}}^5.
		\end{align*}
		Using the bound for the linear flow $v$
		\begin{align}\nonumber
			\|u(s)\|_{H^{s}}=\left\|\mathrm{e}^{sA}y_{0}\right\|_{H^{s}}\leq\| y_{0}\|_{H^{s}},
		\end{align}
		we obtain
		\begin{equation}
			\left\| \frac{\partial g}{\partial s} \right\|_{H^s} \leq C \|y_0\|_{H^{s+11}}^5.
		\end{equation}
		For the derivative with respect to $r$, we have
		\begin{align*}
			\frac{\partial g}{\partial r}=\mathrm{e}^{(\Delta t-s)A}dB(u(s))[\mathrm{e}^{(s-r)A}(-AB(u(s))+ dB(u(r))Au(r)) ],
		\end{align*}
		and each term can be estimated as
		\begin{align*}
			\|dB(u(s))[\mathrm{e}^{(s-r)A}&AB(u(s))]\|_{H^s}\\&=\big\| -\varepsilon \left( 2u(s) u_x(s) \left(\mathrm{e}^{(s-r)A} A B(u(r)) \right) + u^2(s) \partial_x \left( \mathrm{e}^{(s-r)A} A B(u(r)) \right) \right)\\&
			+ \epsilon(1 - 2u(s)) \left( \mathrm{e}^{(s-r)A} A B(u(r)) \right)\big\|_{H^s} \nonumber \\
			&\leq C  \|u(s)\|_{H^{s+7}}^5, 
		\end{align*}
		\begin{align*}
			\|dB(u(s))[\mathrm{e}^{(s-r)A}(dB(u(r))&Au(r))]\|_{H^s}\\=&\big\|-\varepsilon \left( 2u(s) u_x(s) \left( \mathrm{e}^{(s-r)A} (dB(u(r))Au(r)) \right) + u^2(s) \partial_x \left(\mathrm{e}^{(s-r)A} (dB(u(r))Au(r)) \right)\right)\\&
			+ \epsilon(1 - 2u(s)) \left( \mathrm{e}^{(s-r)A} (dB(u(r))Au(r)) \right)
			\big\|_{H^{s}} \nonumber \\
			&\leq C  \|u\|_{H^{s+7}}^5.
		\end{align*}
		Using the available bounds, we obtain
		\begin{align*}
			\left\|\frac{1}{2}{\Delta t}^2g(\frac{1}{2}\Delta t,\frac{1}{2}\Delta t) - \int^{\Delta t}_{0}\int^{s}_{0} g(s,r)dr ds\right\|_{H^s} \leq C \varepsilon^2 (\Delta t)^{3}\|y_{0}\|^{5}_{H^{s+11}}
		\end{align*}
		It remains to estimate the fifth and sixth terms in the local error expression \eqref{locerror}. From \eqref{cont_eqn}, we have
		\begin{align*}
			\|\mathrm{e}^{\frac{\Delta tA}{2}}\left(d^2B(B, B) + dBdBB\right)\left(\Phi^{\theta\Delta t}_{B}(u)\right)\|_{H^s}\leq K\|\Phi^{\theta\Delta t}_{B}(u)\|^4_{H^s}.
		\end{align*}
		To estimate the sixth term, we obtain
		\begin{align*}
			\|dG_{s, r}(u)[v]\|_{H^s}&\leq\|d^2B(\mathrm{e}^{(s-r)A}u)[\mathrm{e}^{(s-r)A}v, \mathrm{e}^{(s-r)A}B(u)]\|_{H^s}\\
			&+\|dB(\mathrm{e}^{(s-r)A}u)[\mathrm{e}^{(s-r)A}(dB(u)[v])]\|_{H^s}, 
		\end{align*}
		where we introduce
		\begin{align*}
			u=\mathrm{e}^{(s-r)A}y(\tau) \,\, \mbox{and}\,\, v=\mathrm{e}^{(r-\tau)s}B(y(\tau)).
		\end{align*}
		Then
		\begin{align*}
			\|dG_{s, r}(u)[v]\|_{H^s}
			\leq  C \|u(s)\|_{H^{s+11}}^5.
		\end{align*}
	The desired result follows.
	\end{proof}
	\subsection{Global error in $H^s$ space}
	\label{strang_global}
	We obtain a global error bound by combining the stability and the local error bounds.
	\begin{theorem}
		\label{global}
		Suppose that the exact solution $y(\cdot,t)$ of Equation~(\ref{KBF}) is in $H^{s}$ for $0\leq t\leq T$. Then the Strang splitting solution $y_n$ given in Equation~(\ref{strng}) has second-order global error for $\Delta t <\bar{\Delta t}$ and $t_n=n\Delta t\leq T$,
		\begin{align*}
			\|y_n-y(\cdot,t_n)\|_{H^s}\leq G {\Delta t}^2,
		\end{align*}
		where $G$ only depends on $\|y_0\|_{H^{s}}$, $\alpha$ and $T.$
	\end{theorem}
	\begin{proof} 
		In the proof we employ the Lady Windermere's fan argument to derive the global error bound from the local consistency and stability properties. The required stability estimate and local error bound are provided in Lemma~\ref{lem_stability2} and Lemma~\ref{localexist2}, respectively. Let
		\[
		y(t_n)=\Phi^{(n-k)\Delta t}(y(t_k))
		\]
		denote the exact solution of Equation~\eqref{KBF} at time $t_n$, starting from the value $y(t_k)$ at time $t_k$. The numerical approximation generated by the Strang splitting scheme satisfies the recurrence relation
		\begin{equation}\nonumber
			y_n=\Psi^{\Delta t}(y_{n-1})
			=
			\Phi^{\Delta t/2}_A \circ \Phi^{\Delta t}_B \circ \Phi^{\Delta t/2}_A (y_{n-1}),
			\qquad n=1,2,\ldots.
		\end{equation}
		
		We compare the numerical flow with the exact evolution and estimate the error in the $H^s$ norm. By a telescopic expansion of the global error, we obtain
		\begin{align}
			\|y_n-y(\cdot,t_n)\|_{H^s}
			&\le
			\sum_{k=0}^{n-1}
			\|
			\Psi^{(n-k-1)\Delta t}
			(
			\Phi^{\Delta t}(y(t_k))
			-
			\Psi^{\Delta t}(y(t_k))
			)
			\|_{H^s}
			\nonumber \\
			&\le
			\sum_{k=0}^{n-1}
			\mathrm{e}^{L(n-k-1)\Delta t}
			\|
			\Phi^{\Delta t}(y(t_k))
			-
			\Psi^{\Delta t}(y(t_k))
			\|_{H^s}.
			\nonumber
		\end{align}
		
		Using the stability bound together with the local error estimate, and noting that $n\Delta t \le T$, we arrive at
		\begin{align*}
			\|y_n-y(\cdot,t_n)\|_{H^s}
			&\le
			\sum_{k=0}^{n-1}
			\mathrm{e}^{LT} C \Delta t^3
			\\
			&\le
			n \mathrm{e}^{LT} C \Delta t^3
			\\
			&\le
			T \mathrm{e}^{LT} C \Delta t^2.
		\end{align*}
		This establishes the second-order convergence in time.
	\end{proof}
	\section{Convergence Analysis of the Strang Splitting Combined with the Fourier Collocation Method}
	In this section, we begin by deriving a semi-discrete scheme and then proceed to construct a fully discrete scheme. Afterwards, the convergence of the fully discrete scheme is investigated in a step-by-step manner.
	\subsection{Semi-discrete collocation scheme of the Fifth-Order KdV-Burgers-Fisher
		Equation}
	In this section, the solution $y(x,t)$  corresponding to Equation 
	(\ref{KBF}) is approximated by trigonometric basis functions for every fixed $t$, as given below:
	\begin{eqnarray}
		y(x,t) \simeq   y_N(x,t)=\sum_{k=-N/2}^{N/2-1} \hat y_{k}(t)e^{ikx},
	\end{eqnarray}
	where $-\pi \leq x \leq \pi,$ \( N \in \mathbb{N} \), $\hat y_{k}(t)$ represents the unknown Fourier coefficients, which are computed using the Fourier collocation method presented in \cite{LubichQuantum,trefethen20}. Let $B_N=\{e^{ikx}: -N/2\leq k \leq N/2-1\}$ be the function space. Fourier interpolation operator, $I_N:C(-\pi, \pi)\rightarrow B_N,$ is expressed by
	\begin{eqnarray}
		I_Ny(x_l)=y(x_l),  \,\ l=-N/2,\cdots, N/2-1,
	\end{eqnarray}
	where $x_l=\frac{2\pi l}{N},\,\ l=-N/2,\cdots, N/2-1$ are any collocation points. A Fourier-interpolated function $y(x)$ can be represented as follows:
	\begin{eqnarray}
		I_Ny(x)=\sum_{k=-N/2}^{N/2-1} \hat y_{k}e^{ikx},
	\end{eqnarray}
	where $\{\hat y_{k}\}_{k=-N/2}^{N/2-1}=\mathcal{F}(y(x_l))_{l=-N/2}^{N/2-1}$ and $\mathcal{F}$ is discrete Fourier transform. 
	
	The semi-discrete Fourier collocation algorithm for Equation (\ref{KBF}) is 
	\begin{align}
		\label{semi_KBF}
		y_{N,t}(x_l,t)
		&=\nu y_{N,xx}(x_l,t)
		- \mu y_{N,xxx}(x_l,t)
		+ \gamma y_{N,xxxxx}(x_l,t) \nonumber \\
		&- \varepsilon I_N\bigl(y_N^2(x_l,t)\, y_{N,x}(x_l,t)\bigr)
		+ \epsilon I_N\bigl(y_N(x_l,t)\bigl(1-y_N(x_l,t)\bigr)\bigr), \nonumber \\
		y_{N}(x_l,0)
		&=I_N y_0(x_l).
	\end{align}
	where $y_N(x_l,t)\in B_N,$ for 
	$-\frac{N}{2}\leq l \leq \frac{N}{2}-1$ . 
	\subsection{A fully discrete collocation scheme for the fifth-order KdV-Burgers-Fisher equation}
	In this section, a fully discrete approximation of Equation~(\ref{KBF}) is developed by employing the Strang splitting technique for the temporal discretization and a Fourier collocation approach for the spatial variables. Hence, the semi-discrete Equation (\ref{semi_KBF}) is split into linear and nonlinear parts as follows: 
	\begin{align}
		\label{decomp_KBF}
		y_{N,t}
		&=\nu y_{N,xx}
		- \mu y_{N,xxx}
		+ \gamma y_{N,xxxxx}, \\
		y_{N,t}&= -\varepsilon I_N\bigl(y_N^2\, y_{N,x}\bigr)
		+ \epsilon I_N \bigl(y_N\bigl(1-y_N\bigr) \bigr). \ \
	\end{align}
	Let us assume that the fully discrete solution is $y_N^n$ at time $t_n=n\Delta t$ with time step $\Delta t$, $\Omega_N^{\Delta t}$ is the full propagator, $\Omega_{BN}^{\Delta t}$ is the full discrete propagator for the nonlinear part after one time step of $\Delta t,$ and 
	$y_{N}^0=I_N (y(\cdot,0))$. Hence, the fully discrete scheme is as follows
	\begin{align}
		\label{strng_four_alg}
		y_N^{n+1}=
		\Omega_N^{\Delta t}(y_N^{n})=e^\frac{ {\Delta t(\nu \partial^2_{x}-\mu \partial^3_{x} + \gamma \partial^5_{x}})}{2} \Omega_{BN}^{\Delta t} e^\frac{ {\Delta t(\nu \partial^2_{x}-\mu \partial^3_{x} + \gamma \partial^5_{x}})}{2} y_N^{n}. 
	\end{align}
	Here, the exponential method is utilized for the solution of the linear part while any suitable approximate method can be utilized for nonlinear part. 
	\subsection{Convergence Analysis of the Fully-discrete Scheme }
	In this section, we first introduce the Fourier interpolation lemmas. Then we develop  stability, consistency and convergence error bounds for the fully discrete algorithm (\ref{strng_four_alg}).
	
	Let $y_N^{n}$ be the fully discrete solution and $y(\cdot,t_n)$ be the analytical solution at time $t_n=n \Delta t$ of Equation (\ref{KBF}). In order to study the fully discrete error, we start with the following error bound
	\begin{align}
		\label{first_bnd}
		\|y_N^{n}-y(\cdot,t_n)\|_{H^s} \leq \|y_N^{n}-y_n\|_{H^s}+ \|y_n-y(\cdot,t_n)\|_{H^s},
	\end{align}
	where $y_n$ is the time-discrete solution given in Section \ref{Section_time}. The second term on the right-hand side of inequality (\ref{first_bnd}) is the time-discrete global error bound given in Theorem \ref{global}. The first term on the right-hand side of inequality (\ref{first_bnd}) can be written as follows
	\begin{align}
		\label{scnd_bnd}
		\|y_N^{n}-y_n\|_{H^s} \leq \|y_N^{n}-I_N y_n\|_{H^s} + \|I_N y_n-y_n\|_{H^s},
	\end{align}
	where $I_N$ is the Fourier interpolation operator.
	
	Here, lemmas related to the Fourier interpolation operator are presented.
	\begin{lemma}
		\label{lemma1}
		Let $y$ be a $2\pi$ periodic function such that $\partial^m_xy\in L^2,$ for some $m\geq1.$ Then the interpolation error is bounded in $L^2$ by
		\begin{eqnarray}
			\|y-I_Ny\|_{L^2}\leq G N^{-m}\|\partial^m_x y\|_{L^2},
		\end{eqnarray}
		where $G$ depends only on $m$. And for any integer $d\geq1$,
		\begin{eqnarray}
			\|\partial_x^{d}(y-I_Ny)\|_{L^2}\leq GN^{-m}\|\partial^{m+d}_x y\|_{L^2}.
		\end{eqnarray}
	\end{lemma}
	\begin{proof}
		See \cite[Theorem 1.7]{LubichQuantum}.
	\end{proof}
	\begin{lemma}
		\label{lemma2}
		Let $y$ be a $2\pi$ periodic function such that $\partial^m_x y\in H^s,$ for some $m\geq1.$ Then the interpolation error is bounded in $H^s$ by
		\begin{eqnarray}
			\label{eqn15}
			\|y-I_Ny\|_{H^s}\leq GN^{-m}\|y\|_{H^{m+s}},
		\end{eqnarray}
		where $G$ depends only on $m$. And for any integer $d\geq 1,$
		\begin{eqnarray}
			\label{eqn16}
			\|\partial_x^{d}(y-I_Ny)\|_{H^s}\leq GN^{-m}\|y\|_{H^{m+s+d}}.
		\end{eqnarray}
	\end{lemma}
	\begin{proof}
		See \cite[Lemma 2.1]{Gucuyenen2026}.
	\end{proof}
	\begin{lemma} [Stability of Fourier interpolation]
		\label{lemma3}
		Let $y$ and $v$ be $2\pi$ periodic functions such that $y, v\in H^s$, with $s>0.$ Then the Fourier interpolation operator satisfies the following 
		\begin{align}
			\label{inequality1}
			\left\| I_N (y) \right\|_{H^s}
			&\le
			K_0\left\| y \right\|_{H^s}, \\
			\label{inequality2}
			\left\| I_N (yv) \right\|_{H^s}
			&\le
			K_1\left\| I_N y \right\|_{H^s} \left\| I_N v \right\|_{H^s}, \ \
		\end{align}
		where $K_0$ and $K_1$ are independent of $N.$
	\end{lemma}
	\begin{proof}
		Using the boundedness of the Fourier interpolation operator in $H^s$ together with the algebra property of $H^s$ for $s>0$, we obtain inequalities (\ref{inequality1}) and (\ref{inequality2}), see \cite{LubichQuantum,trefethen20}.
	\end{proof}
	\begin{lemma}
		\label{lemma4}
		Let $y,v \in H^s(-\pi,\pi)$. Then there exists a constant $G>0$ such that
		\begin{equation}\label{lemma:INproduct}
			\|\partial_x I_N(yv)\|_{H^s}
			\le
			G\,\|I_N y\|_{H^s}\,\|I_N v\|_{H^s}.
		\end{equation}
	\end{lemma}
	\begin{proof}
		We recall that the Fourier interpolation of the product $yv$ can be written as
		\[
		I_N(yv)(x)
		=
		\sum_{k=-\frac{N}{2}}^{\frac{N}{2}-1} h_k e^{ikx},
		\]
		where
		$
		h_k = \mathcal{F}\big(y(x_j)v(x_j)\big)
		=
		\sum_{\ell=-\frac{N}{2}}^{\frac{N}{2}-1}
		\hat{y}_{k-\ell}\,\hat{v}_\ell,$
		$-\frac{N}{2}\le j \le \frac{N}{2}.$
		Starting from the left-hand side of \eqref{lemma:INproduct}, we have
		\[
		\|\partial_x I_N(yv)\|_{H^s}
		\le
		\sum_{j=0}^{s}
		\|\partial_x^{\,j+1} I_N(yv)\|_{L^2}.
		\]
		We first consider the highest-order term:
		\begin{align*}
			\|\partial_x^{\,s+1} I_N(yv)\|_{L^2}^2
			&=
			\left\|
			\partial_x^{\,s+1}
			\left(
			\sum_{k=-\frac{N}{2}}^{\frac{N}{2}-1}
			h_k e^{ikx}
			\right)
			\right\|_{L^2}^2
			\\
			&=
			\left\|
			\sum_{k=-\frac{N}{2}}^{\frac{N}{2}-1}
			(ik)^{s+1}
			\sum_{\ell=-\frac{N}{2}}^{\frac{N}{2}-1}
			\hat{y}_{k-\ell}\hat{v}_\ell
			e^{ikx}
			\right\|_{L^2}^2.
		\end{align*}
		Using the Cauchy--Schwarz inequality, we obtain
		\begin{eqnarray}
			\|\partial_x^{\,s+1} I_N(yv)\|_{L^2}^2
			&\le
			\sum_{k=-\frac{N}{2}}^{\frac{N}{2}-1}
			|k|^2
			\max_{k,\ell}
			\left|
			\frac{k^s}{(k-\ell)^s\,\ell^s}
			\right|
			\sum_{\ell=-\frac{N}{2}}^{\frac{N}{2}-1}
			|(k-\ell)^s \hat{y}_{k-\ell}|^2
			\sum_{\ell=-\frac{N}{2}}^{\frac{N}{2}-1}|\ell^s \hat{v}_\ell|^2, \nonumber \\
			&\le
			G
			\left(
			\sum_{\ell=-\frac{N}{2}}^{\frac{N}{2}-1}
			|(k-\ell)^s \hat{y}_{k-\ell}|^2
			\right)^{1/2}
			\left(
			\sum_{\ell=-\frac{N}{2}}^{\frac{N}{2}-1}
			|\ell^s \hat{v}_\ell|^2
			\right)^{1/2},\nonumber \ \
		\end{eqnarray}
		where the constants
		$\max_{k,\ell}
		\left|
		\frac{k^s}{(k-\ell)^s\,\ell^s}
		\right|$
		\text{and}
		$\sum_{k=-\frac{N}{2}}^{\frac{N}{2}-1} |k|^2$
		are bounded by positive constants $G_1$ and $G_2$, respectively, and $G=G_1\cdot G_2.$
		Hence,
		\[
		\|\partial_x^{\,s+1} I_N(yv)\|_{L^2}
		\le
		G\,\|I_N y\|_{H^s}\,\|I_N v\|_{H^s}.
		\]
		The lower-order terms can be treated similarly, which completes the proof.
	\end{proof}
	\subsection{Stability Analysis}
	\begin{lemma}[Regularity of the nonlinear discrete operator]
		\label{regu_lemma1}
		If $\|y_0\|_{H^s} \le R$,
		then there exists $\tilde{t}(R)>0$ such that
		$\|\Omega_{BN}^t(y_0)\|_{H^s} \le 2R$,
		$0<t\le \tilde{t}(R)$.
	\end{lemma}
	\begin{proof}
		The full discrete scheme is given in Equation (\ref{strng_four_alg}).
		Let $z(t)=\Omega_{BN}^t(y_0)$. Then we compute
		\[
		\frac{1}{2}\frac{d}{dt}\|z\|_{H^s}^2
		=
		(z,z_t)_{H^s}.
		\]
		Using the definition of the $H^s$ inner product, we obtain
		\begin{eqnarray}
			\frac{1}{2}\frac{d}{dt}\|z\|_{H^s}^2
			&=&
			\sum_{j=0}^s
			\int_{\mathbb{R}}
			\partial_x^j z\,
			\partial_x^j z_t
			\,dx
			\nonumber \\
			&=&
			\sum_{j=0}^s
			\int_{\mathbb{R}}
			\partial_x^j z\,
			\partial_x^j
			\Big(
			-\varepsilon I_N(z^2 \partial_x z)
			+
			\epsilon I_N(z(1-z))
			\Big)
			\,dx. \ \
		\end{eqnarray}
		Hence,
		\begin{eqnarray}
			\frac{1}{2}\frac{d}{dt}\|z\|_{H^s}^2
			&=
			\sum_{j=0}^s
			\int_{\mathbb{R}}
			\partial_x^j z\,
			\partial_x^j
			\big(
			-\varepsilon I_N(z^2 \partial_x z)
			\big)
			\,dx \nonumber 
			\\
			&\quad
			+
			\sum_{j=0}^s
			\int_{\mathbb{R}}
			\partial_x^j z\,
			\partial_x^j
			\big(
			\epsilon I_N(z(1-z))
			\big)
			\,dx.
		\end{eqnarray}
		For each $0\le j\le s$, the above terms can be bounded as follows:
		\begin{eqnarray}
			\label{regu_bound1_2}
			&\Big|
			\int_{\mathbb{R}}
			\partial_x^j z\,
			\partial_x^j
			\big(
			-\varepsilon I_N(z^2 \partial_x z)
			\big)
			\,dx
			\Big|
			+
			\Big|
			\int_{\mathbb{R}}
			\partial_x^j z\,
			\partial_x^j
			\big(
			\epsilon I_N(z(1-z))
			\big)
			\,dx
			\Big| \nonumber \\
			&\leq 
			\int_{\mathbb{R}}
			|\partial_x^j z\,
			\partial_x^j
			\big(
			-\varepsilon I_N(z^2 \partial_x z)
			\big)|
			\,dx +
			\int_{\mathbb{R}}
			| \partial_x^j z\,
			\partial_x^j
			\big(
			\epsilon I_N(z(1-z))
			\big)|
			\,dx.
		\end{eqnarray}
		For the first term of (\ref{regu_bound1_2}), with the help of Equations (\ref{inequality1})-(\ref{lemma:INproduct}), we have 
		\[
		\begin{aligned}
			\label{first_part_regufull}
			\int_{\mathbb{R}}
			|\partial_x^j z\,
			\partial_x^j
			\big(
			-\varepsilon I_N(z^2 \partial_x z)
			\big)|
			\,dx
			&\le
			\varepsilon \left\| \partial_x^j z\right\|_{L^2} \left\| \partial_x^j I_N(z^2 \partial_x z)\right\|_{L^2} \nonumber \\
			&\le
			\varepsilon G \left\| z\right\|_{H^{j}} \left\|  I_N(z^2 )\right\|_{H^{j-1}} \left\|  I_N(\partial_x z)\right\|_{H^{j-1}} \nonumber \\
			&\le
			\varepsilon G K_0 K_1\left\| z\right\|_{H^j} \left\|  I_N(z)\right\|_{H^j} \left\|  I_N(z)\right\|_{H^j} \left\| \partial_x z \right\|_{H^{j-1}} \nonumber \\
			&\le
			\varepsilon G K^3_0 K_1 \left\| z\right\|_{H^j} \left\|  z\right\|_{H^j} \left\| z\right\|_{H^j} \left\| z \right\|_{H^{j}} \\
			&\le G_1 \left\| z \right\|^4_{H^{j}.} \\
		\end{aligned}
		\]
		For the second term of (\ref{regu_bound1_2}), we have 
		\[
		\begin{aligned}
			\label{second_part_regufull}
			\int_{\mathbb{R}}
			| \partial_x^j z\,
			\partial_x^j
			\big(
			\epsilon I_N(z(1-z))
			\big)|
			\,dx
			&\le
			\epsilon \left\| \partial_x^j z\right\|_{L^2} \left\| \partial_x^j I_N(z(1-z) )\right\|_{L^2} \nonumber \\
			&\le
			\epsilon G \left\| z\right\|_{H^j} \left\| I_N(z)\right\|_{H^{i-1}} \left\| I_N(1-z)\right\|_{H^{i-1}} \nonumber \\
			&\le
			\epsilon G K^2_0\left\| z\right\|_{H^j} \left\| z\right\|_{H^{i}} \left\| 1-z\right\|_{H^{i}} \nonumber \\
			&\le G_2 ( \left\| z\right\|^2_{H^{i}}+ \left\| z\right\|^3_{H^{i}} ).  \\
		\end{aligned}
		\]
		Combining the bounds of the first and the second term of Equation (\ref{regu_bound1_2}) yields
		\begin{eqnarray}
			\label{denk39}
			\|z\|_{H^s}\frac{d}{dt}\|z\|_{H^s}
			&\le G_3 (\left\| z\right\|^2_{H^{s}}+ \left\| z\right\|^3_{H^{s}} + \left\| z\right\|^4_{H^{s}} ).
		\end{eqnarray}
		For $\left\|z\right\|_{H^s}\geq1,$ Equation (\ref{denk39}) can be simplified to 
		\begin{eqnarray}
			\label{denk40}
			\frac{d}{dt}\|z\|_{H^s}
			&\le G_4 ( \left\| z\right\|^3_{H^{s}} ).
		\end{eqnarray}
		Solving Equation (\ref{denk40}) results in $\|z\|_{H^s}
		\le \frac{\|y_0\|_{H^s}}{\sqrt{1-2G_4t\|y_0\|^2_{H^s}}}$ and by choosing $\tilde{t}$ such that $1-2G_4tR^2=1/4,$ the desired result is obtained. 
	\end{proof}
	\begin{lemma}[Stability]
		\label{stability_full}
		Let $y_0, \widetilde{y}_0 \in H^s$. Then 
		\begin{align}
			\left\|\Omega_{N}^{\Delta t}(y_0)-\Omega_{N}^{\Delta t}(\widetilde{y}_0)\right \|_{H^s}
			&\le  
			e^{K_s\Delta t}
			\left\|y_0-\widetilde{y}_0\right\|_{H^s}, \ \
		\end{align}
		where $K_s=K_1 K^2_0\,\varepsilon (M^2_s+1+M_s)$ with $\sup_{0\leq t\leq \Delta t} 
		(\|\Omega_{BN}^{\Delta t}(y_0)\|_{H^s} +\|\Omega_{BN}^{\Delta t}(\widetilde{y}_0)\|_{H^s} ) \leq M_s$.
	\end{lemma}
	\begin{proof}
		We know that linear operator $e^\frac{ {\Delta t(\nu \partial^2_{x}-\mu \partial^3_{x} + \gamma \partial^5_{x}})}{2}$ is preserved in the $H^s$ norm. Hence it is sufficient to bound the nonlinear part. 
		Let $\Omega_{BN}^{\Delta t}(y_0)=\Lambda(\Delta t)$ and $\Omega_{BN}^{\Delta t}(\widetilde{y}_0)=\Gamma(\Delta t),$ where  $\Delta t$ is the time step. 
		Consider the following equations satisfying the nonlinear part:
		\begin{align}
			\partial_t \Lambda
			&= -\varepsilon\, I_N\!\left(\Lambda^2\,\partial_x \Lambda\right)
			+ \epsilon\, I_N\!\left(\Lambda(1-\Lambda)\right),
			\qquad \Lambda(0)=y_0, \label{eq:Lambda}
			\\
			\partial_t \Gamma
			&= -\varepsilon\, I_N\!\left(\Gamma^2\,\partial_x \Gamma\right)
			+ \epsilon\, I_N\!\left(\Gamma(1-\Gamma)\right),
			\qquad \Gamma(0)=\tilde y_0. \label{eq:Gamma}
		\end{align}
		After subtracting \eqref{eq:Gamma} from \eqref{eq:Lambda} with rearranging, we obtain 
		\begin{align}
			\partial_t(\Lambda-\Gamma)
			&= -\frac{\varepsilon}{3}\,
			I_N\!\left((\Lambda-\Gamma)(\Lambda^2+\Lambda\Gamma+\Gamma^2)\right)
			+\epsilon\, I_N\!\left((\Lambda-\Gamma)\bigl(1-(\Lambda+\Gamma)\bigr)\right).
			\label{eq:diff}
		\end{align}
		After rearranging \eqref{eq:diff}, taking the Sobolev norm $H^s,$  
		and applying the relevant Lemma \ref{lemma3}, we obtain
		\begin{align}
			\|\Lambda(\cdot,t)-\Gamma(\cdot,t)\|_{H^s}
			&\le
			\|y_0-\tilde y_0\|_{H^s}
			\notag\\
			&\quad
			+ K_1\,\varepsilon
			\int_0^t
			\Big\|
			I_N\!\big(\Lambda(\cdot,\xi)-\Gamma(\cdot,\xi)\big)
			\Big\|_{H^s}
			\notag\\
			&\qquad\qquad\times
			\Big\|
			I_N\!\big(\Lambda^2(\cdot,\xi)
			+ \Lambda(\cdot,\xi)\Gamma(\cdot,\xi)
			+ \Gamma^2(\cdot,\xi)\big)
			\notag\\
			&\qquad\qquad\quad
			+ I_N \bigl(1-(\Lambda(\cdot,\xi)+\Gamma(\cdot,\xi))\bigr)
			\Big\|_{H^s}
			\, d\xi, \nonumber \\
			&\le
			\|y_0-\tilde y_0\|_{H^s}
			\notag\\
			&\quad
			+ K_1 K^2_0\,\varepsilon
			\int_0^t
			\|
			\Lambda(\cdot,\xi)-\Gamma(\cdot,\xi)
			\|_{H^s}
			\notag\\
			&\qquad\qquad\times
			\|
			\Lambda^2(\cdot,\xi)
			+ \Lambda(\cdot,\xi)\Gamma(\cdot,\xi)
			+ \Gamma^2(\cdot,\xi)
			\notag\\
			&\qquad\qquad\quad
			+ \bigl(1-(\Lambda(\cdot,\xi)+\Gamma(\cdot,\xi))\bigr)
			\|_{H^s}
			\, d\xi, \nonumber \\
			\label{eq:stability}
		\end{align}
		with $\epsilon \le\varepsilon.$
		Assume that $\sup_{0\leq t\leq \Delta t} 
		(\|\Omega_{BN}^{\Delta t}(y_0)\|_{H^s} +\|\Omega_{BN}^{\Delta t}(\widetilde{y}_0)\|_{H^s} ) \leq M_s$, since $\Omega_{BN}^{\Delta t}(y_0)$ is bounded by Lemma \ref{regu_lemma1}. Hence after rearranging Equation (\ref{eq:stability}), we get 
		\begin{align}
			\|\Lambda(\cdot,t)-\Gamma(\cdot,t)\|_{H^s}
			&\le
			\|y_0-\tilde y_0\|_{H^s}
			\notag\\
			&\quad
			+ K_1 K^2_0\,\varepsilon (M^2_s+1+M_s)
			\int_0^t
			\|
			\Lambda(\cdot,\xi)-\Gamma(\cdot,\xi)
			\|_{H^s} \, d\xi
			\notag\\
			&\le
			e^{K_s t}\|y_0-\tilde y_0\|_{H^s},
			\label{eq:stability2}
		\end{align}
		where $K_s=K_1 K^2_0\,\varepsilon (M^2_s+1+M_s),$ and the proof is completed.
	\end{proof}
	\begin{lemma}[Local error]
		\label{full_loc}
		If $y_0\in H^{m+s+5},$ then 
		\small{
			\begin{eqnarray}
				\label{loc_formul}
				\|\Omega_{N}^{\Delta t}(I_N(y_0))-I_N(\Psi^{\Delta t}(y_{0}))\|_{H^s} \leq G \Delta t N^{-m}\big( e^{K_h \Delta t} \|y_0\|_{H^{m+s+5}} +
				\| \Psi^{\Delta t}(y_{0}) \|_{H^{m+s+5}} \big),
		\end{eqnarray}}
		where $G$ depends only on $m$, $K_h=K_1 K^2_0\,\varepsilon (M^2_h+1+M_h)$ with $\sup_{0 \le t \le \Delta t} 
		(\|\Omega_{N}^{\Delta t}(y_0)\|_{H^s} +\| \Psi^{\Delta t}(y_{0})\|_{H^s} ) \leq M_h$.
	\end{lemma}
	\begin{proof}
		In order to estimate the local truncation error, we begin by examining the linear term of Equation 
		(\ref{KBF}) and subsequently address the nonlinear part. Assume that the full discretized and the time discretized approximate solutions be $\Omega_{N}^{\Delta t}(y_0)=\Lambda$ and $\Psi^{\Delta t}(y_{0})=\Gamma$, respectively, satisfying the following
		\begin{align}
			\label{loc_err_linear1}
			\partial_t \Lambda
			&=(\nu \partial_{xx}
			- \mu \partial_{xxx}
			+ \gamma \partial_{xxxxx})\Lambda,
			\qquad \Lambda(0)=y_0, \\
			\partial_t \Gamma
			\label{loc_err_linear2}
			&=(\nu \partial_{xx}
			- \mu \partial_{xxx}
			+ \gamma \partial_{xxxxx})\Gamma,
			\qquad \Gamma(0)=\tilde y_0. \ \
		\end{align}
		After subtracting the Fourier interpolation of Equation (\ref{loc_err_linear2}) from Equation (\ref{loc_err_linear1}), we obtain
		\begin{align}
			\label{lineardefect}
			\partial_t (\Lambda-I_N(\Gamma))
			&=(\nu \partial_{xx}
			- \mu \partial_{xxx}
			+ \gamma \partial_{xxxxx})(\Lambda-I_N(\Gamma))+ \delta_t. \ \
		\end{align}
		Here, $\delta_t=(\nu \partial_{xx}
		- \mu \partial_{xxx}
		+ \gamma \partial_{xxxxx})(I_N(\Gamma)-\Gamma)$ is the defect operator and taking Sobolev inner product of (\ref{lineardefect}) with respect to $ \eta_N=\Lambda-I_N(\Gamma)$ yields the following 
		\begin{align}
			\label{linear_inner_pro}
			(\partial_t  \eta_N,  \eta_N )_{H^s}=((\nu \partial_{xx}
			- \mu \partial_{xxx}
			+ \gamma \partial_{xxxxx})\eta_N, \eta_N)_{H^s} + (\delta_t,\eta_N)_{H^s}.
		\end{align}
		The first term appearing on the right-hand side of equation (\ref{linear_inner_pro}) can be expressed as
		\begin{align}
			\label{linear_inner_real}
			Re((\nu \partial^2_{x}
			- \mu \partial^3_{x}
			+ \gamma \partial^5_{x})\eta_N, \eta_N)_{H^s}
			&=Re((-\nu\xi^2+\mu (i\xi^3)+ \gamma (i\xi^5))\hat{\eta}_N, \hat{\eta}_N)_{H^s}, \nonumber \\
			&=Re\big(\sum_{j=0}^{s}\int (-\nu\xi^2+\mu (i\xi^3)+ \gamma (i\xi^5))|(i\xi)^{j}\hat{\eta}_N|^2d\xi \big), \nonumber \\
			&=\sum_{j=0}^{s}Re\int (-\nu\xi^2+\mu (i\xi^3)+ \gamma (i\xi^5))|(i\xi)^{j}\hat{\eta}_N|^2d\xi , \nonumber \\
			&\leq 0,
		\end{align}
		since $Re(-\nu\xi^2+\mu (i\xi^3)+ \gamma (i\xi^5))\leq 0$ with $\nu>0.$ Therefore, the first term of Equation (\ref{linear_inner_pro}) is neglected.  Applying Lemma \ref{lemma2} to the second term on the right-hand side of (\ref{linear_inner_pro}) yields the following expression
		\begin{align}
			\label{second}
			(\delta_t,\eta_N)_{H^s}
			&
			\leq 
			\| (\nu \partial^2_{x}
			- \mu \partial^3_{x}
			+ \gamma \partial^5_{x})(I_N (\Gamma) -\Gamma)\|_{H^s}\|\eta_N\|_{H^s} \nonumber \\
			&
			\leq GN^{-m} \|\Gamma\|_{H^{m+s+5}}\|\eta_N\|_{H^s}. 
		\end{align}
		Hence substituting (\ref{second}) into 
		(\ref{linear_inner_pro}) and after integrating from $0$ to $\Delta t,$ we obtain 
		\begin{align}
			\label{linear_son}
			\|\Lambda(\cdot,t)-I_N(\Gamma(\cdot,t))\|_{H^s}\leq \|y_0-I_N(\tilde y_0)\|_{H^s}+ G N^{-m}\Delta t \|\Gamma(\cdot,t)\|_{H^{m+s+5}}
		\end{align}
		and $\|\Gamma(\cdot,0)\|_{H^{m+s+5}}=\|\tilde y_0\|_{H^{m+s+5}}.$ Turning to the nonlinear component, the fully discretized and time discretized  solutions satisfy 
		\begin{align}
			\label{non1}
			\partial_t \Lambda
			&= -\varepsilon\, I_N\!\left(\Lambda^2\,\partial_x \Lambda\right)
			+ \epsilon\, I_N\!\left(\Lambda(1-\Lambda)\right),
			\qquad \Lambda(0)=y_0,  \\
			\label{non2}
			\partial_t \Gamma
			&= -\varepsilon\, \Gamma^2\,\partial_x \Gamma
			+ \epsilon\,\Gamma(1-\Gamma),
			\qquad \Gamma(0)=\tilde y_0. \ \
		\end{align}
		Taking the Fourier interpolation of (\ref{non2}), subtracting it from (\ref{non1}), using Lemma \ref{lemma2} and displaying similar argument to Lemma \ref{stability_full} yields the following 
		\begin{align}
			\label{non_son}
			\|\Lambda(\cdot,t)-I_N(\Gamma(\cdot,t))\|_{H^s}\leq e^{K_h \Delta t} \|y_0-I_N(\tilde y_0)\|_{H^s},
		\end{align}
		where $K_h=K_1 K^2_0\,\varepsilon (M^2_h+1+M_h)$ with $\sup_{0\leq t\leq \Delta t} 
		(\|\Lambda\|_{H^s} +\|\Gamma\|_{H^s} ) \leq M_h$.
		Remember that the full discretized solution, $\|\Lambda\|_{H^s}$ is bounded by Lemma \ref{regu_lemma1} and the time discretized solution  $\|\Gamma\|_{H^s}$ is bounded by Lemma \ref{strng_regularity}. Finally, combining the linear component (\ref{linear_son}) and the nonlinear component (\ref{non_son}), we obtain the desired estimate (\ref{loc_formul}).
	\end{proof}
	\subsection{Global error bound of fully discrete scheme}
	In order to obtain the global error bound of the fully discretized scheme, Lady Windermere's fan argument, which combines the local error bound given in Lemma \ref{full_loc} with stability result given in Lemma \ref{stability_full}, will be used.
	In what follows, we assume that
	\[
	\rho_1(n)
	=
	\max_{\substack{j=0,\ldots,n-1 \\ i=0,\ldots,n-j-1}}\| (\Omega_{N}^{\Delta t})^{i}(I_N y_j)\|_{H^s},
	\] 
	\[
	\rho_2(n)
	=
	\max_{\substack{j=1,\ldots,n}}\{\| \Omega_{N}^{\Delta t}(y_{j-1})\|_{H^s},\| \Psi^{\Delta t}(y_{j-1})\|_{H^s} \},
	\]
	\[
	\rho_3(n)=\rho_1(n)+\rho_2(n),
	\]
	covering the entire set of points in Lady Windermere's fan argument over $[t_0, t_{n-1}].$ Then the fan produces
	\begin{align}
		\|y_N^n-I_N(y_n)\|_{H^s}
		&=
		\|(\Omega_{N}^{\Delta t})^{n}(I_N(y_0))-I_N(y_n) \|_{H^s}, \nonumber \\
		&\leq \sum_{j=0}^{n-1} \|(\Omega_{N}^{\Delta t})^{n-j-1} (\Omega_{N}^{\Delta t}(I_N y_j)) -(\Omega_{N}^{\Delta t})^{n-j-1} (I_N( \Psi^{\Delta t}(y_j)) \|_{H^s}, \nonumber \\
		&\leq \sum_{j=0}^{n-1} e^{H\rho_1(n)\Delta t(n-j-1)} 
		G\Delta t N^{-m}(e^{H\rho_2(n)\Delta t} \|y_j\|_{H^{m+s+5}} + \|y_{j+1}\|_{H^{m+s+5}}),\nonumber \\
		&\leq \sum_{j=0}^{n-1} e^{H\rho_3(n)\Delta t(n-j-1)} 
		G\Delta t N^{-m}(\|y_j\|_{H^{m+s+5}} + \|y_{j+1}\|_{H^{m+s+5}}).\nonumber \\
	\end{align}
	From Lemma \ref{strng_regularity}, we have $\|y_{j+1}\|_{H^{m+s+5}}\leq P $ for $j\Delta t \leq T.$ Therefore, we construct the following estimate 
	\begin{align}
		\label{no59}
		\|y_N^n-I_N(y_n)\|_{H^s}
		&\leq P N^{-m}\left( \frac{e^{H\rho_3(n)n\Delta t }-1}{H\rho_3(n)}\right),
	\end{align}
	where $P$ depends on $m,$ $s,$ $T,$ $\sup_{0\leq t\leq T}\|y(\cdot,\tau)\|_{H^{m+s+5}}.$ Then rearranging (\ref{no59}) with the help of the Lemma \ref{lemma2} yields 
	\begin{align}
		\label{no60}
		\|y_N^n-y_n\|_{H^s}
		&\leq P N^{-m}\left( \frac{e^{H\rho_3(n)n\Delta t }-1}{H\rho_3(n)} + 1\right).
	\end{align}
	We need to eliminate $\rho_3(n)$ from (\ref{no60}) as it includes an approximate solution. Hence, suppose that $\rho_3(n)\le 2P$ where $P$ is given in (\ref{no60}) with $n \Delta t\le T$ for large $N.$   The inequality (\ref{no60}) is satisfied for $\|(\Omega_{N}^{\Delta t})^{n-j}(I_N(y_j))-y_n \|_{H^s}$ with $j=0,\cdots,n.$ Then for each $j$ we get
	\begin{align}
		\label{no61}
		\|(\Omega_{N}^{\Delta t})^{n-j}(I_N(y_j))\|_{H^s}\leq \| y_n\|_{H^s} + 
		P N^{-m}\left( \frac{e^{2P T }-1}{2P} + 1\right),
	\end{align}
	which leads to $\rho_3(n+1)\le 2P$ for large $N$ by induction. We conclude that Equation (\ref{no59}) can be written as follows
	\begin{align}
		\label{no62}
		\|  y_N^n-I_N(y_n) \|_{H^s} \leq PN^{-m}.
	\end{align}
	In the final step, by substituting Equation (\ref{no62}) together with the time error bound given in Theorem \ref{global} into Equation (\ref{first_bnd}), we obtain the result stated in Theorem \ref{theo_full_glob}.
	\begin{theorem}[Global error]
		\label{theo_full_glob}
		Let the exact solution of the fifth-order Korteweg-de Vries-Burgers-Fisher equation (\ref{KBF}), $y=y(x,t),$ be in $H^{m+s+5}$ for $t\in [0,T]$ and $x\in [-\pi, \pi]$. Then the fully discrete solution (\ref{strng_four_alg}) has the following error bound
		\begin{align}
			\label{full_global_theorem}
			\| y_N^{n}-y(\cdot,t_n) \|_{H^s} \leq P (N^{-m} + \Delta t^2),
		\end{align}
		where $t_n=n \Delta t,$ $P$ depends on m, s, T and $M_y=\max_{0\leq t \leq T}\|y(\cdot, t)\|_{H^{m+s+5}}.$
	\end{theorem}
	
	\section{Numerical Results}
	For the numerical confirmation of the convergence results, Equation (\ref{KBF}) is split into two sub-equations
	\begin{align}
		y_t(x,t)&=\nu y_{xx}(x,t)- \mu y_{xxx}(x,t)+ \gamma y_{xxxxx}(x,t), \label{KBFsplt1} \\
		y_t(x,t)&=- \varepsilon y^2(x,t) y_x(x,t)+ \epsilon y(x,t)\bigl(1-y(x,t)\bigr), \label{KBFsplt2}
	\end{align}
	which are then solved successively.
	
	We apply the discrete Fourier transform (DFT) to discretize the subproblems (\ref{KBFsplt1}) and (\ref{KBFsplt2}) in the spatial variable. Consider Equation (\ref{KBFsplt1}) in the interval $x\in[0,2\pi]$ for $t>0$ together with periodic boundary conditions. In this case, the solution $y(x,t)$ can be represented by the Fourier series \cite{trefethen20}
	\begin{eqnarray} \label{eq:fourier}
		y(x,t)=\sum_{k=-\infty}^{\infty} \hat y_{k}(t)e^{ikx},
	\end{eqnarray}
	where $\hat y_{k}$ denotes the Fourier coefficients corresponding to the initial data.
	
	On a uniform grid with spacing $h=\frac{2\pi}{N}$, the discrete Fourier transform (DFT) is defined as
	\begin{eqnarray} \label{eq:dft}
		\hat y_{k}(t)=h\sum_{j=1}^{N} y_{j}(t)e^{-ikx_{j}}, \qquad  k=-\frac{N}{2}+1, \ldots , \frac{N}{2},
	\end{eqnarray}
	while the inverse discrete Fourier transform (IDFT) is given by \cite{trefethen20}
	\begin{eqnarray} \label{eq:idft}
		y_{j}(t)=\frac{1}{2\pi}\sum_{k=-N/2+1}^{N/2} \hat y_{k}(t)e^{ikx_{j}}, \qquad  j=1, \ldots , N.
	\end{eqnarray}
	
	We employ the Fast Fourier Transform (FFT) for computation of the DFT \cite{trefethen20}. Applying the DFT to Equation (\ref{KBFsplt1}) yields
	\begin{eqnarray}
		\label{RBfourier}
		\frac{d}{dt}\hat{y}_{k}(t)=(-\nu k^2 + i\mu k^3 - i\gamma k^5 )\hat{y}_{k}(t),\qquad \hat y_{k}(0)= \hat y_{0k},
	\end{eqnarray}
	and the solution is given as
	\begin{equation}\label{eq:slndftburg}
		\hat y_{k}(t)= e^{(-\nu k^2 + i\mu k^3 - i\gamma k^5 ) t}\hat y_{0k}, \qquad  k=-N/2+1, \ldots , N/2. \nonumber
	\end{equation}
	
	We now apply the DFT to Equation (\ref{KBFsplt2}) and then we have
	\begin{align}\label{eq:dftKBF21}
		\frac{d}{dt}\hat y_k(t) 
		&= -\frac{\varepsilon}{3} ik \, 
		\F\Bigl((\F^{-1}(\hat y_k(t)))^{3}\Bigr) \nonumber\\
		&\quad + \epsilon \, \mathcal{F}\Bigl(\F^{-1}(\hat y_k(t))\Bigr) \nonumber\\
		&\quad - \epsilon \, \F\Bigl((\F^{-1}(\hat y_k(t)))^{2}\Bigr).
	\end{align}
	Equation (\ref{eq:dftKBF21}) can also be written as 
	\begin{eqnarray} \label{eq:dftcons1}
		\frac{d}{dt} Y=F(Y),
	\end{eqnarray}
	where $ k=-N/2+1, \ldots , N/2$, $\hat y_{k}(t)=Y$, $F(Y)=-\frac{\varepsilon}{3}ik\,\F\!\left((\F^{-1}(\hat y_k(t)))^{3}\right)
	+\epsilon \mathcal{F}\left(\mathcal{F}^{-1}(\hat y_k(t))\right)
	-\epsilon \F\!\left((\F^{-1}(\hat y_k(t)))^{2}\right).$
	We approximate Equation (\ref{eq:dftcons1}) with the fourth order Runge-Kutta method as follows
	\begin{eqnarray}
		a&=&F(Y_{n})\nonumber\\
		b&=&F(Y_{n}+\frac{\Delta t}{2}k_{1})\nonumber\\
		c&=&F(Y_{n}+\frac{\Delta t}{2}k_{2})\nonumber\\
		d&=&F(Y_{n}+\Delta t k_{3})\nonumber\\
		Y_{n+1}&=&Y_{n}+\frac{1}{6}\Delta t(a+2b+2c+d)
	\end{eqnarray}
	where $\Delta t$ is the time step. On the other hand, in order to obtain reference solutions numerically, we employ the method of integrating factors with the classical fourth order Runge-Kutta method applied to Equation (\ref{KBF}).
	\paragraph*{\bf Example} Let us examine Equation (\ref{KBF}) with the initial condition defined as 
	\begin{equation} \label{eq:iintcond1}
		y(x,0)=\frac{1}{2}+\frac{1}{4}\sin(x),
	\end{equation}
	under the assumption of periodic boundary conditions in the space domain $[0, 2\pi]$. The parameters are chosen as $\nu = 1$, $\mu = 1.0$, $\varepsilon = 1$, $\epsilon = 1$, and $\gamma = 1.0$. We perform the spatial discretization using $N=256$ collocation points. The error at the final time is measured using the 2-norm error. In Fig.~\ref{fig:ex1}, we show the performance of the Strang splitting method by plotting the solution error versus the number of steps. The observed convergence orders are presented in Table \ref{tab:horder_obs}.  The lines in Fig.~\ref{fig:ex1} and the values reported in Table~\ref{tab:horder_obs} verify the theoretical results of second-order convergence.
	
	\begin{figure}[H]
		\centering
		\pgfplotsset{every axis plot/.append style={line width=1.0pt, mark size=2pt},
			tick label style={font=\footnotesize},
			every axis/.append style={%
				minor x tick num=1,
				minor y tick num=4,
				minor z tick num = 2,
				scale only axis, 
				font=\footnotesize
			}
		}
		\setlength\figurewidth{.50\textwidth}
		\setlength\figureheight{.45\textwidth}
		\includegraphics[width=0.7\textwidth]{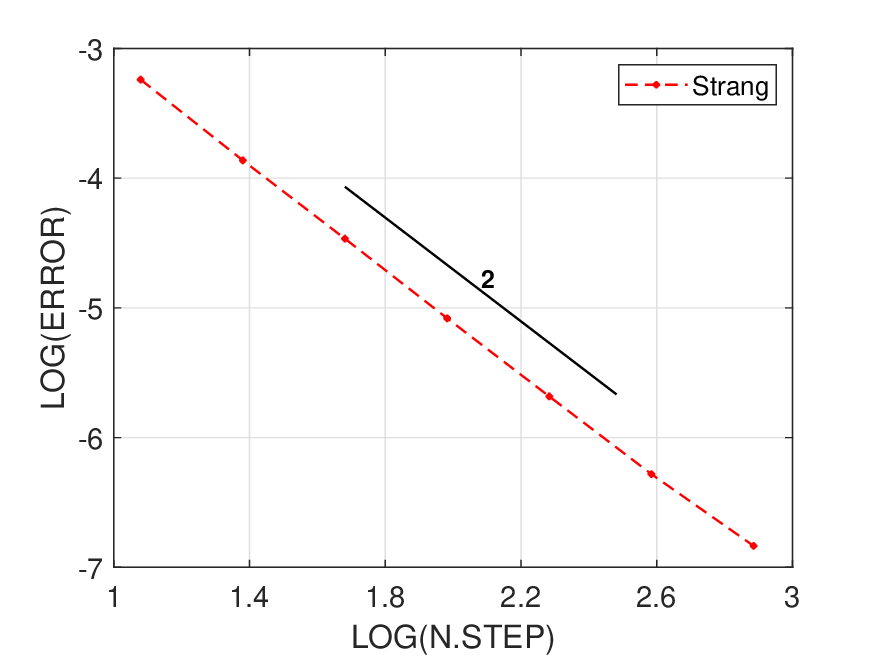}
		\caption{\label{fig:ex1} Solution error versus the number of steps at $t=1$.  }
	\end{figure}

	\begin{table*}[htbp]
		\caption{Convergence orders of the Strang splitting method}
		\centering
		\begin{tabular}{cc}
			\MyHead{2.3cm}{T/$\Delta t$} & \MyHead{2.3cm}{Strang} \\ [1ex]
			\cline{1-2}
			24  & 2.0693 \\
			48  & 2.0063 \\
			96  & 2.0381 \\
			192 & 2.0024 \\
			384 & 1.9899 \\ [1ex]
			\hline
		\end{tabular}
		\label{tab:horder_obs}
	\end{table*}
	
	\section{Conclusion}
	A numerical approximation scheme has been constructed efficiently by combining the Fourier collocation method in spatial discretization with the Strang splitting technique in time integration for the fifth-order Korteweg-de Vries-Burgers-Fisher equation (KBF). In the first part of the paper, the convergence of the time discrete scheme, which is constructed by Strang splitting, has been proved with the help of numerical quadratures using Lie commutator bounds under suitable regularity assumptions. Then, the convergence of the full discrete scheme has been successfully established by analyzing the global time discretization error and the Fourier interpolation error. Finally, the computational findings demonstrate that the proposed scheme ensures spatial accuracy together with second order temporal accuracy.

	%
	%
	%
	
\end{document}